\documentclass[reqno]{amsart}

\usepackage[utf8]{inputenc} 
\usepackage{amssymb}
\usepackage{mathrsfs}
\usepackage{graphicx}
\usepackage{latexsym}
\usepackage{stmaryrd}
\usepackage{enumitem}
\usepackage[bookmarks,hyperfootnotes=false]{hyperref}
\usepackage{multirow}
\usepackage{xcolor}
\usepackage{caption}
\colorlet{darkishRed}{red!80!black}
\colorlet{darkishBlue}{blue!60!black}
\colorlet{darkishGreen}{green!60!black}
\hypersetup{
    draft = false,
    bookmarksopen=true,
    colorlinks,
    linkcolor={red!60!black},
    citecolor={green!60!black},
    urlcolor={blue!60!black}
}
\usepackage[msc-links]{amsrefs} 
\usepackage{doi}

\renewcommand{\PrintDOI}[1]{\doi{#1}}

\let\setminus=\smallsetminus
\usepackage{tikz}
\usepackage{tikz-cd}
\usetikzlibrary{calc,through,intersections,arrows, trees, positioning, decorations.pathmorphing, cd}
\usepackage{relsize}
\usepackage{comment}
\usepackage{svg}
\usepackage{mathtools}
\usepackage{nccmath}
\usepackage{pifont}

\let\setminus=\smallsetminus

\DeclareMathOperator{\medcup}{\mathsmaller{\bigcup}}

\renewcommand{\subset}{\subseteq}

\newcommand{ \N } { \mathbb{N} }

\newcommand{ \Q } { \mathbb{Q} }

\makeatletter

\def\calCommandfactory#1{%
   \expandafter\def\csname c#1\endcsname{\mathcal{#1}}}
\def\frakCommandfactory#1{%
   \expandafter\def\csname frak#1\endcsname{\mathfrak{#1}}}

\newcounter{ctr}
\loop
  \stepcounter{ctr}
  \edef\X{\@Alph\c@ctr}
  \expandafter\calCommandfactory\X
  \expandafter\frakCommandfactory\X
  \edef\Y{\@alph\c@ctr}
  \expandafter\frakCommandfactory\Y
\ifnum\thectr<26
\repeat

\renewcommand{\cC}{\mathscr{C}}

\renewcommand{\cP}{\mathscr{P}}

\setenumerate{label={\normalfont (\roman*)},itemsep=0pt}

\def\lowfwd #1#2#3{{\mathop{\kern0pt #1}\limits^{\kern#2pt\raise.#3ex
\vbox to 0pt{\hbox{$\scriptscriptstyle\rightarrow$}\vss}}}}
\def\lowbkwd #1#2#3{{\mathop{\kern0pt #1}\limits^{\kern#2pt\raise.#3ex
\vbox to 0pt{\hbox{$\scriptscriptstyle\leftarrow$}\vss}}}}

\def\ve{\kern-1.5pt\lowfwd e{1.5}2\kern-1pt}
\def\ev{\kern-1pt\lowbkwd e{0.5}2\kern-1pt}
\def\vf{\kern-2pt\lowfwd f{2.5}2\kern-1pt}

\def\vSd{{\mathop{\kern0pt S\lower-1pt\hbox{${}
     \scriptstyle'$}}\limits^{\kern2pt\raise.1ex
     \vbox to 0pt{\hbox{$\scriptscriptstyle\rightarrow$}\vss}}}}

\newtheorem{theorem}{Theorem}[section] 

\newtheorem{lemma}[theorem]{Lemma}

\newtheorem{mainresult}{Theorem}

\newtheorem{innercustomthm}{Theorem}

\newenvironment{customthm}[1]
  {\renewcommand\theinnercustomthm{#1}\innercustomthm}
  {\endinnercustomthm}

\theoremstyle{definition}

\theoremstyle{remark}

\newcommand{\iecon}{infi\-nite\-ly edge-con\-nec\-ted}
\newcommand{\TTT}{T_{\aleph_0}\!\ast t}
\def\Pigraph{$\Pi$-graph}
\newcommand{\FG}{F}
\newcommand{\hFG}{\breve{F}}

\newcommand{\uG}[1]{G_{\ge{#1}}}

\newcommand{\dG}[1]{G_{\le{#1}}}

\usepackage{caption}

\begin{document}
\vspace*{-2.7cm} 
\title{Ubiquity and the Farey graph}
\author{Jan Kurkofka}
\address{University of Hamburg, Department of Mathematics, Bundesstraße 55 (Geomatikum), 20146 Hamburg, Germany}
\email{jan.kurkofka@uni-hamburg.de}
\@namedef{subjclassname@2020}{\textup{2020} Mathematics Subject Classification}

\begin{abstract}
We construct a countable planar graph which, for any two vertices $u,v$ and any integer $k\ge 1$, contains $k$ edge-disjoint order-compatible $u$--$v$ paths but not infinitely many.
This graph has applications in Ramsey theory, in the study of connectivity and in the characterisation of the Farey graph.
\end{abstract}

\maketitle

\vspace*{-.8cm}
\begin{figure}[h]
    \centering
    \includegraphics[width=.5\textwidth]{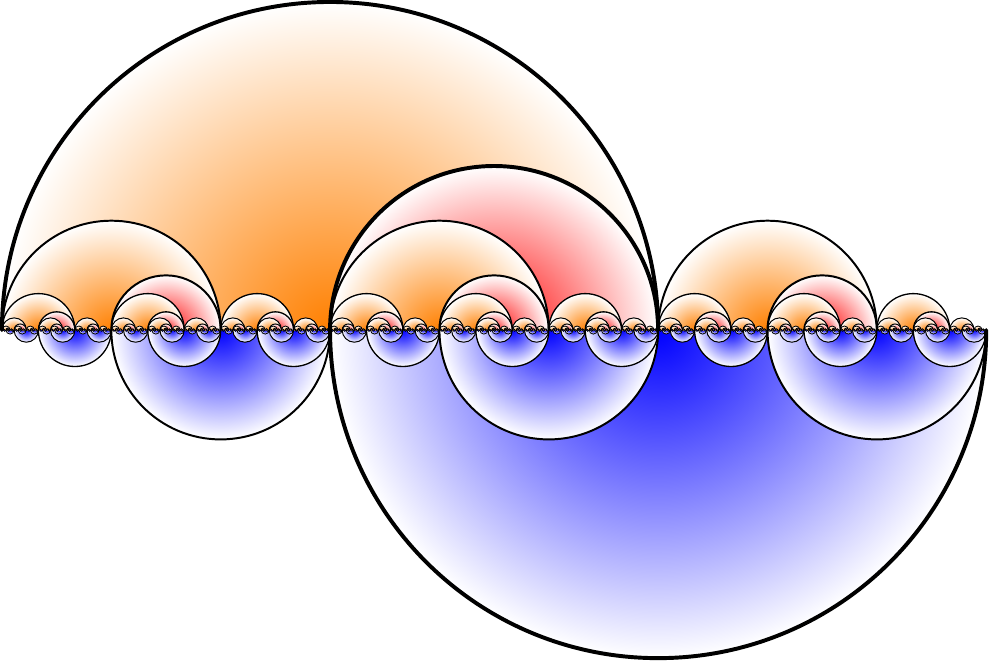}
    \caption{The whirl graph, coloured}
    \label{fig:Whirl}
\end{figure}

\vspace*{-0.4cm}\section{Introduction}

\noindent ``One of the most basic problems in an infinite setting that has no finite equivalent is whether or not `arbitrarily many', in some context, implies `infinitely many'.'' 
(Diestel~\cite{DiestelBook5}).
For example, Halin~\cites{H65,DiestelBook5} proved that if a graph contains $k$ disjoint rays for every integer $k$, then it contains infinitely many disjoint rays.
Substructures of a given type---subgraphs, minors, rooted minors or whatever---of which there must exist infinitely many disjoint copies (for some notion of disjointness) in a given graph as soon as there are arbitrarily (finitely) many such copies are called \emph{ubiquitous}~\cite{DiestelBook5}.
Examples of ubiquity results can be found in
\cites{%
A77,
A79,
A02,
A13,
DoubleRayEUbiquitous,%
U1,U2,U3,%
DiestelBook5,%
H65,
H70,
L76,
W76
}.

Usually, ubiquity problems are trivial as soon as the substructures considered are finite.
For example, if a graph $G$ contains $k$ disjoint $u$--$v$ paths for every integer~$k$ and some fixed vertices $u$ and~$v$, we can greedily find infinitely many disjoint \mbox{$u$--$v$} paths in~$G$.
Similarly, edge-disjoint paths between two fixed vertices are clearly ubiquitous.
Interestingly, this changes as soon as we require our edge-disjoint paths to traverse their common vertices in the same order.

Let us call two $u$--$v$ paths \emph{order-compatible} if they traverse their common vertices in the same order.
Our first aim in this paper is to show that edge-disjoint order-compatible paths between two given vertices are not ubiquitous: we shall construct a graph $G$, the whirl graph shown in Figure~\ref{fig:Whirl},
that has two vertices $u$ and $v$ such that $G$ contains $k$ edge-disjoint order-compatible $u$--$v$ paths for every integer $k$, but not infinitely many.
In fact, the whirl graph $G$ will have this property for \emph{all} pairs of vertices:


\begin{mainresult}\label{Mainresult}
The whirl graph is a countable planar graph that contains $k$ edge-disjoint pairwise order-compatible paths between every two of its vertices for every $k\in\N$, but which does not contain infinitely many edge-disjoint pairwise order-compatible paths between any two of its vertices.
\end{mainresult}

\subsection*{Applications}

Our result has two applications.

\noindent\begin{figure}[ht]
\centering
\begin{minipage}{.45\textwidth}
  \centering
  \includegraphics[height=.7944\textwidth]{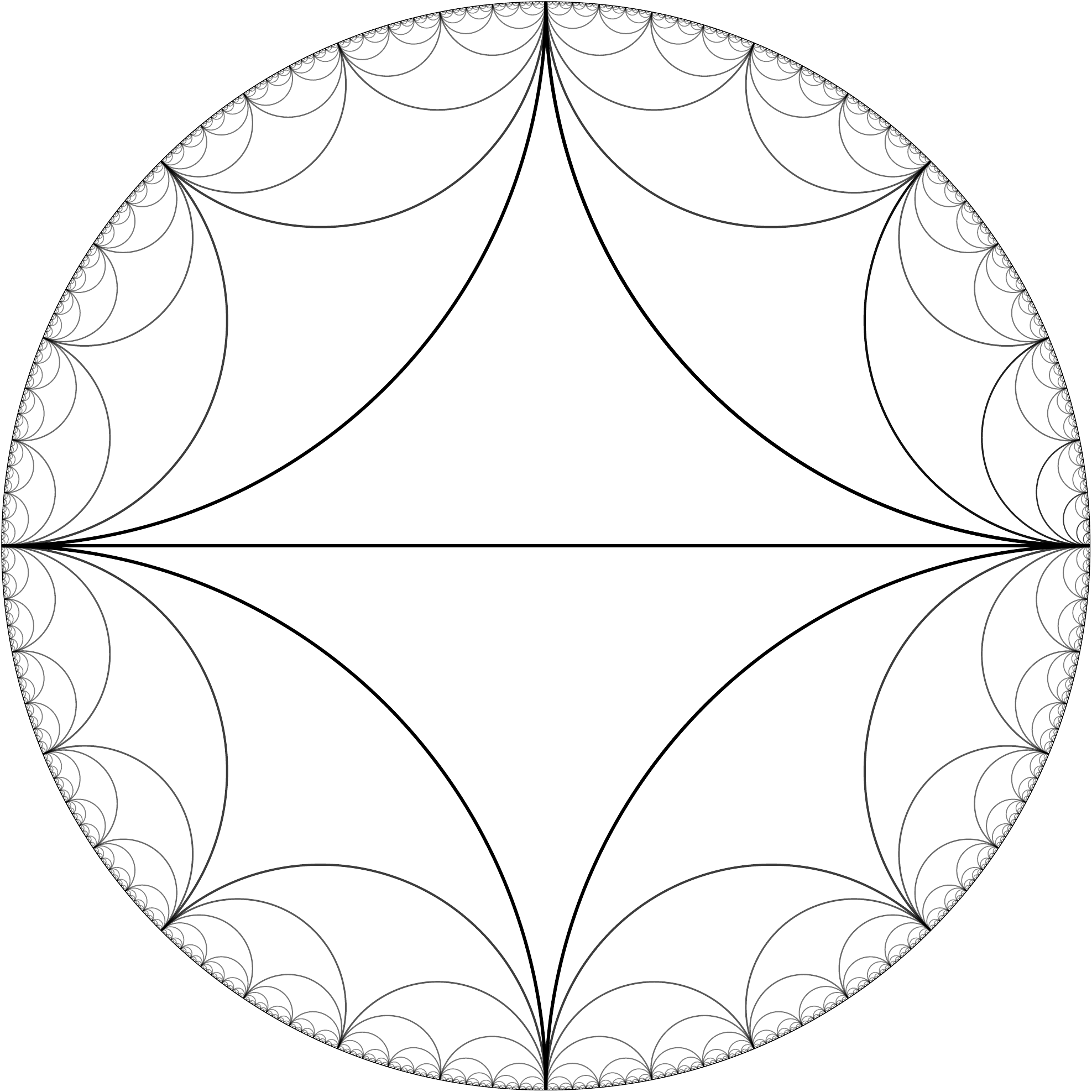}
  \captionof{figure}{The Farey graph}
  \label{fig:FareyGraph}
\end{minipage}%
\begin{minipage}{.55\textwidth}
  \centering
  \includegraphics[height=.65\textwidth]{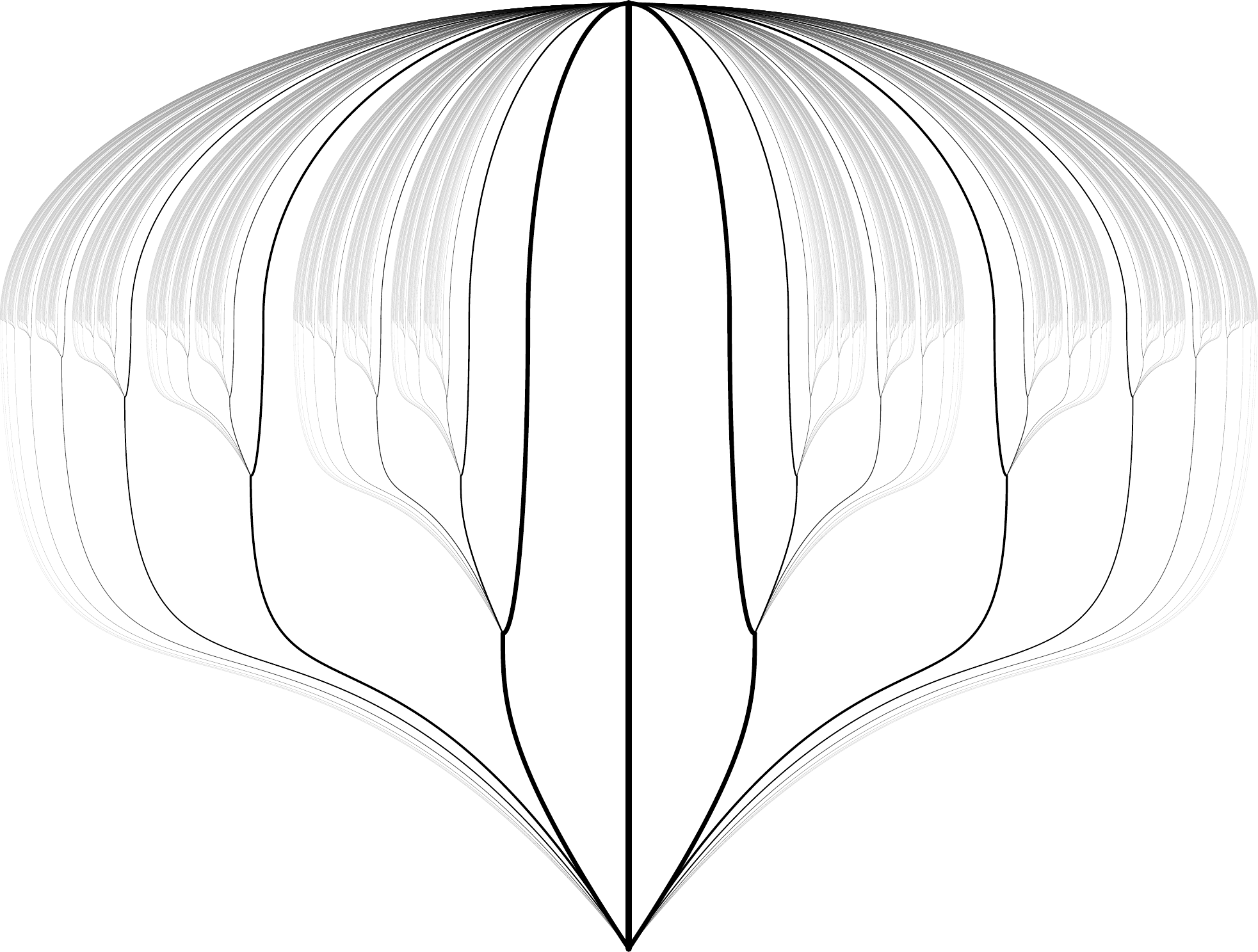}
  \captionof{figure}{The graph $\TTT$}
  \label{fig:TTT}
\end{minipage}
\end{figure}

The Farey graph, shown in Figure~\ref{fig:FareyGraph} and surveyed in \cites{OfficeHoursGroupTheory,hatcher2017topology}, plays a role in a number of mathematical fields ranging from group theory and number theory to geometry and dynamics~\cite{OfficeHoursGroupTheory}.
Curiously, graph theory has not been among these until very recently, when it was shown that the Farey graph plays a central role in graph theory too: it is one of two infinitely edge-connected graphs that must occur unavoidably as a minor in every infinitely edge-connected graph.
The second graph is $\TTT$, the graph obtained from the infinitely-branching tree $T_{\aleph_0}$ by joining an additional vertex $t$ to all its vertices; see Figure~\ref{fig:TTT}.

\begin{customthm}{\cite{TypicalInfinitelyEdgeconnectedGraphs}}
Every infinitely edge-connected graph contains either the Farey graph or $T_{\aleph_0}\!\ast t$ as a minor.
\end{customthm}

\noindent This result lies in the intersection of Ramsey theory and the study of connectivity; see the introduction of~\cite{TypicalInfinitelyEdgeconnectedGraphs}.
Related results can be found in~\cites{DiestelBook5,GenGridTheorem,GollinHeuerKcon,halin78,JoerisPhD,OporowskiOxleyThomas}; see \cite{DiestelBook5}*{§9.4} or the introduction of~\cite{GollinHeuerKcon} for surveys.

The obvious question this theorem raises is whether it is best possible in the sense that one cannot replace `minor' with `topological minor' in its wording.
The whirl graph and Theorem~\ref{Mainresult} are needed in~\cite{TypicalInfinitelyEdgeconnectedGraphs} to answer this question in the affirmative:

\begin{customthm}{\cite{TypicalInfinitelyEdgeconnectedGraphs}}
The whirl graph is \iecon\ but contains neither the Farey graph nor $\TTT$ as a topological minor.
\end{customthm}

The second application of the whirl graph and Theorem~\ref{Mainresult} concerns the first graph-theoretic characterisation of the Farey graph.
Very recently it was shown that the Farey graph
is uniquely determined by its connectivity: up to minor-equivalence, the Farey graph is the unique minor-minimal graph that is \iecon\ but such that every two vertices can be finitely separated.
A~\mbox{\emph{\Pigraph }} is an \iecon\ graph that does not contain infinitely many independent paths between any two of its vertices.
A \Pigraph\ is \emph{typical} if it occurs as a minor in every \Pigraph .
Note that any two typical \Pigraph s are minors of each other; we call such graphs \emph{minor-equivalent}.

\begin{customthm}{\cite{FareyGraphChar}}
Up to minor-equivalence, the Farey graph is the unique typical \mbox{\Pigraph }.
\end{customthm}

\noindent This theorem continues to hold if we require all minors to be tight:
A \emph{tight} minor is a minor with finite branch sets.
A \Pigraph\ is \emph{tightly} typical it it occurs as a tight minor in every \Pigraph .
Note that any two tightly typical \Pigraph s are tight minors of each other; we call such graphs \emph{tightly minor-equivalent}.

\begin{customthm}{\cite{FareyGraphChar}}
Up to tight minor-equivalence, the Farey graph is the unique tightly typical \Pigraph .
\end{customthm}

\noindent The obvious question this theorem raises is whether it is best possible in the sense that one cannot replace `tight' with `topological'.
The whirl graph and Theorem~\ref{Mainresult} are needed in~\cite{FareyGraphChar} to answer this question in the affirmative:

\begin{customthm}{\cite{FareyGraphChar}}
The whirl graph is a \Pigraph\ that contains the Farey graph as a tight minor but not as a topological minor.
\end{customthm}

\noindent This theorem in turn raises the two questions how exactly the Farey graph is contained in the whirl graph as a minor and how large the branch sets actually are.
We shall use the Cantor set to explicitly determine a Farey graph minor in the whirl graph with branch sets of size two; see Section~\ref{sec:FareyInWhirl} for the explicit description of the Farey graph minor.

\begin{customthm}{\ref{FareyGraphCantor}}
The whirl graph contains the Farey graph as a minor with branch sets of size two, but not as a topological minor.
\end{customthm}

This note is organised as follows.
We introduce the whirl graph in Section~\ref{sec:ThmOne} where we also prove Theorem~\ref{Mainresult}, and we prove Theorem~\ref{FareyGraphCantor} in Section~\ref{sec:FareyInWhirl}.

\section{Proof of Theorem~1}\label{sec:ThmOne}

\noindent We use the graph-theortic notation of Diestel's book~\cite{DiestelBook5}.
A \emph{separation} of a graph $H$ is a set $\{A,B\}$ such that $A\cup B=V(H)$ and $H$ contains no edge between $A\setminus B$ and $B\setminus A$.
Then $A\cap B$ is the \emph{separator} of $\{A,B\}$.

The \emph{whirl graph}, shown in Figure~\ref{fig:Whirl}, is the graph $G=(V,E)$ on $V:=\bigcup_{n=1}^\infty V_n$ where $V_n:=\big\{\tfrac{0}{3^n},\tfrac{1}{3^n},\ldots,\tfrac{3^n}{3^n}\big\}$
and with edge set $E:=\bigcup_{n=1}^\infty E_n$ where
\begin{align*}
    E_n:=\Big\{\,\big\{\tfrac{3k}{3^n},\tfrac{3k+2}{3^n}\big\},\big\{\tfrac{3k+1}{3^n},\tfrac{3k+2}{3^n}\big\},\big\{\tfrac{3k+1}{3^n},\tfrac{3k+3}{3^n}\big\}\;\,\Big\vert\,\;k\in\big\{0,1,\ldots,3^{n-1}-1\big\}\,\Big\}.
\end{align*}
For every integer $n\ge 1$ we define the three subgraphs
\begin{align*}
    \dG{n}:=(V_n,\medcup_{k=1}^nE_k)\quad\text{and}\quad G_n:=(V_n,E_n)\quad\text{and}\quad \uG{n}:=(V,\medcup_{k=n}^\infty E_k);
\end{align*}
see Figure~\ref{fig:WhirlBlock} for an illustration.
Note that $G_n$ is a Hamilton path of $\dG{n}$ for all $n$.

\begin{figure}[ht]
    \centering
    \includegraphics[width=.618\textwidth]{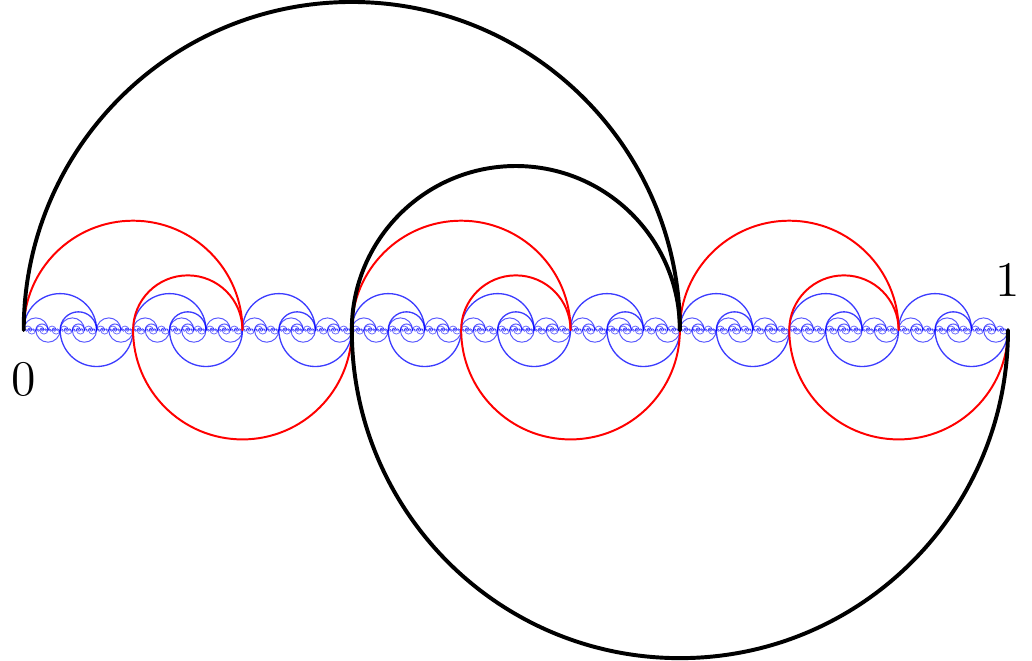}
    \caption{$G_1=\dG 1$ is black, $\dG 2$ is the union of black and red, $G_2$ is red, $\uG 2$ is the union of red and blue, and $\uG 3$ is blue}
    \label{fig:WhirlBlock}
\end{figure}




For the proof of Theorem~\ref{Mainresult} we need another theorem and a lemma.
At the end of one of my talks at Hamburg that involved order-compatible paths, Joshua Erde asked: \emph{Is there a function $f\colon\N\to\N$ such that, for every graph $H$ and every two vertices $u$ and $v$ of $H$, the existence of at least $f(k)$ many edge-disjoint $u$--$v$ paths in $H$ implies the existence of $k$ many edge-disjoint pairwise order-compatible $u$--$v$ paths in $H$?}
The next day, Jakob Kneip answered the question in the affirmative for $f(k)=k$ the identity on~$\N$:

\begin{theorem}[Kneip]\label{edgeDisjPathsIntoPWorderCompatible}
Let $H$ be any graph, let $u$ and $v$ be any two distinct vertices of~$H$, and let $n$ be any natural number.
If $H$ contains $n$ edge-disjoint $u$--$v$ paths, then $H$ also contains $n$ edge-disjoint pairwise order-compatible $u$--$v$ paths.
\end{theorem}

\begin{proof}
Given $H,u,v,n$ we suppose that $H$ contains $n$ edge-disjoint $u$--$v$ paths.
Choose a path-system $\cP$ of $n$ edge-disjoint $u$--$v$ paths in $H$ that uses as few edges of $H$ as possible.
Then the paths in $\cP$ are pairwise order-compatible:
For this, assume for a contradiction that $P$ and $Q$ are paths in $\cP$ such that $P$ traverses two vertices $x$ and $y$ as $x<_P y$ while $Q$ traverses them as $y<_Q x$.
Then $uPx\cup xQv$ and $uQy\cup yPv$ are connected edge-disjoint subgraphs of $P\cup Q$, so we may choose one $u$--$v$ path in each of the two.
Now replacing $P$ and $Q$ with these two new paths yields a system of $n$ edge-disjoint $u$--$v$ paths using strictly fewer edges of~$H$ than~$\cP$, since the edges of $xPy$ and $yQx$ are not used by the new paths (contradiction).
\end{proof}

\begin{lemma}\label{PathControl}
Let $u,v\in V$ be any two vertices with $u<_\Q v$ and let $n>1$ be any integer with $u,v\in V_{n-1}$.
If $P\subset\uG{n}$ is any $u$--$v$ path, then
\begin{align*}
V_{n-1}\cap [u,v]\subset V(P)\subset V\cap [u,v]
\end{align*}
and $P$ traverses the vertices in $V_{n-1}\cap [u,v]$ in the natural order induced by~$\Q$.
\end{lemma}

\begin{proof}
Every vertex $x\in V_{n-1}\setminus\{0,1\}$ is a cutvertex of $\uG{n}$ and the components of $\uG{n}-x$ are $\uG{n}[\,V\cap [0,x)\,]$ and $\uG{n}[\,V\cap (x,1]\,]$.
This clearly implies the statement of the lemma.
\end{proof}

Now we prove Theorem~\ref{Mainresult}:

\begin{proof}[{Proof of Theorem~\ref{Mainresult}}]
Clearly, $G$ is planar.
It is \iecon\ because it can be written as the edge-disjoint union $\bigcup_{n\in\N}G_n=G$.
In particular, it follows from Theorem~\ref{edgeDisjPathsIntoPWorderCompatible} that $G$ contains $k$ edge-disjoint pairwise order-compatible paths between any two vertices, for every $k\in\N$.

It remains to show that $G$ does not contain infinitely many edge-disjoint pairwise order-compatible paths between any two vertices.
For this, let any two vertices $u$ and $v$ of $G$ be given, say with $u<_\Q v$.
We pick any integer $N>1$ such that $u,v\in V_{N-1}$.
Since all the edge sets $E_0,E_1,\ldots$ are finite, it suffices to show the following assertion:
\vspace{.5\baselineskip}
\begin{fleqn}%
\begin{equation*}%
\hspace{\parindent}\begin{aligned}
    \parbox{\textwidth-2\parindent}{\emph{Whenever $P$ is any $u$--$v$ path in $\uG N$ and $M\ge N$ is the minimal integer such that $\dG{M}$ contains $P$, no $u$--$v$ path in $\uG{M+1}$ is order-compatible with~$P$.}}%
\end{aligned}%
\end{equation*}%
\end{fleqn}%

\vspace{.5\baselineskip}
\noindent Let $P$ and $M$ be given.
By the minimality of $M$ the path $P$ must have an edge $e$ in $E_M$.
Let $x$ and $y$ be the two consecutive elements of $V_{M-1}\subset \Q$ bounding an interval $[x,y]$ that contains the endvertices of $e$.
We claim that $P$ contains the subpath $xG_My$ of the $0$--$1$ Hamilton path $G_M$ of $\dG{M}$.

Indeed, on the one hand the edge $e$ lies on $P$, so $P$ has a vertex in $V_M\cap (x,y)$.
On the other hand, the separator of the separation $\{\,V_M\setminus (x,y),V_M\cap [x,y]\,\}$ of $\dG{M}$ is $\{x,y\}$ while $u,v\in V_{M-1}\subset V_M\setminus (x,y)$ and $\dG{M}[\,V_M\cap [x,y]\,]=xG_M y$.
Thus, the $u$--$v$ path $P$ meeting $V_M\cap (x,y)$ implies that $P$ contains both vertices $x$ and $y$ and that either $xPy=xG_My$ or $yPx$ is the reverse of $xG_My$.
In either case, $P$ contains $xG_M y$.

Now let $Q$ be any $u$--$v$ path in $\uG{M+1}$.
We show that $Q$ is not order-compatible with $P$.
For this, we consider the path $xG_M y$ that is contained in~$P$.
We apply Lemma~\ref{PathControl} twice:
First, we apply it to $u,v,N$ and $P$ to establish $V(P)\subset V\cap [u,v]$ which ensures $u\le x<y\le v$ in~$\Q$.
And the second time we apply it to $u,v,M+1$ and $Q$ to establish $V_M\cap [u,v]\subset V(Q)$ and that $Q$ traverses the vertices in $V_M\cap [u,v]$ in the natural order induced by~$\Q$.
Altogether, we deduce that $Q$ traverses the vertices in $V(xG_My)\subset V_M\cap [x,y]\subset V_M\cap [u,v]$ in the natural order induced by~$\Q$.
Since $P$ contains $xG_My$ and $xG_M y$ is the path
\begin{align*}
    \tfrac{3k}{3^M}\,\tfrac{3k+2}{3^M}\,\tfrac{3k+1}{3^M}\,\tfrac{3k+3}{3^M}
\end{align*}
for the appropriate integer $k$, the paths $P$ and $Q$ certainly are not order-compatible, completing the proof.
\end{proof}

\section{Finding the Farey graph in the whirl graph}\label{sec:FareyInWhirl}

\noindent The \emph{Farey graph} $F$ is the graph on $\Q\cup\{\infty\}$ in which two rational numbers $a/b$ and $c/d$ in lowest terms (allowing also $\infty=(\pm 1)/0$) form an edge if and only if $\det\bigl( \begin{smallmatrix}a & c\\ b & d\end{smallmatrix}\bigr)=\pm 1$, cf.~\cite{OfficeHoursGroupTheory}.
In this paper we do not distinguish between the Farey graph and the graphs that are isomorphic to it.
For our graph-theoretic point of view it will be more convenient to work with the following purely combinatorial definition of the Farey graph that is indicated in~\cite{OfficeHoursGroupTheory} and~\cite{hatcher2017topology}.

The \emph{halved Farey graph} $\hFG_0$ of order $0$ is a $K^2$ with its sole edge coloured blue.
Inductively, the \emph{halved Farey graph} $\hFG_{n+1}$ of order $n+1$ is the edge-coloured graph that is obtained from $\hFG_n$ by adding a new vertex $v_e$ for every blue edge $e$ of $\hFG_n$, joining each $v_e$ precisely to the endvertices of $e$ by two blue edges, and colouring all the edges of $\hFG_n\subset\hFG_{n+1}$ black.
The \emph{halved Farey graph} $\hFG:=\bigcup_{n\in\N}\hFG_n$ is the union of all $\hFG_n$ without their edge-colourings, and the \emph{Farey graph} is the union $F=G_1\cup G_2$ of two copies $G_1,G_2$ of the halved Farey graph such that $G_1\cap G_2=\hFG_0$.

It was shown in~\cite{FareyGraphChar} that any graph contains the Farey graph as a minor with finite branch sets if it is \iecon\ and does not contain infinitely many independent paths between any two vertices.
As independent paths are order-compatible, it follows that the whirl graph contains the Farey graph as a minor with finite branch sets.
The result in~\cite{FareyGraphChar}, however, does not
provide an explicit description of the Farey graph minor in the whirl graph, nor does it tell us how large the branch sets actually are.
In our situation, the latter is especially unsatisfactory, as we already know that infinitely many branch sets must be non-trivial because the whirl graph does not contain the Farey graph as a topological minor.
That is why in this section we use the Cantor set to explicitly determine a Farey graph minor in the whirl graph with branch sets of size two.

Recall that the Cantor set is $C:=\bigcap_{n=0}^\infty\bigcup\cC_n$ where $\cC_0:=\{\,[0,1]\,\}$ and $\cC_{n+1}$ is obtained from $\cC_n$ by replacing each interval $[a,a+\Delta]\in\cC_n$ with the two intervals $[a,a+\tfrac{1}{3}\Delta]$ and $[a+\tfrac{2}{3}\Delta,a+\Delta]$.

We define the subgraph $G^\ast:=(C^\ast,E^\ast)\subset G$ on $C^\ast:=\bigcup_{n=1}^\infty C_n^\ast$ and with edge set $E^\ast:=\bigcup_{n=1}^\infty E_n^\ast$ where
\vspace*{-2.5mm}
\begin{align*}
C^\ast_n :={}&\big\{\,a,a+\tfrac{1}{3}\Delta,a+\tfrac{2}{3}\Delta,a+\Delta\;\big\vert\; [a,a+\Delta]\in\cC_{n-1}\,\big\}\\
={}&\{\,x,y\mid [x,y]\in\cC_n\,\}=V_n\cap C\text{ and}\\
E^\ast_n :={}&\big\{\,\{a,a+\tfrac{2}{3}\Delta\},\{a+\tfrac{1}{3}\Delta,a+\tfrac{2}{3}\Delta\},\{a+\tfrac{1}{3}\Delta,a+\Delta\}\;\big\vert\; [a,a+\Delta]\in \cC_{n-1}\,\big\}\subset E_n.
\end{align*}
We shall find the halved Farey graph (minus one edge) as a contraction minor of~$G^\ast$.
For this, we write $G^\ast_{\le n}:=(C^\ast_n,\bigcup_{k=1}^n E_k^\ast)$ and $M:=\bigcup_{n=1}^\infty M_n$ where
\begin{align*}
M_n &:=\big\{\,\{a+\tfrac{1}{3}\Delta,a+\tfrac{2}{3}\Delta\}\;\big\vert\; [a,a+\Delta]\in \cC_{n-1}\,\big\}\subset E_n^\ast;
\end{align*}
see Figure~\ref{fig:WhirlFarey} for an illustration.
We write $M_{\le n}:=\bigcup_{k=1}^n M_k$.
If $D$ is an independent set of edges and $H$ is any graph, then we denote by $H/D$ the contraction minor of $H$ obtained by contracting the edges in $D\cap E(H)$.


\noindent\begin{figure}[ht]
\centering
\begin{minipage}{.5\textwidth}
  \centering
  \includegraphics[width=.95\textwidth]{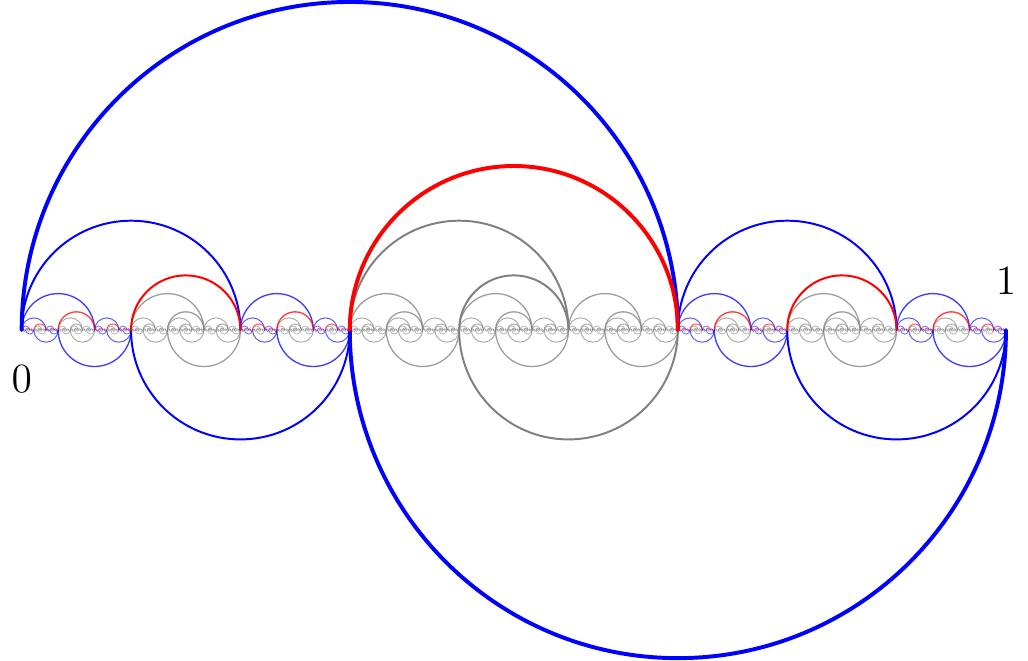}
\end{minipage}%
\begin{minipage}{.5\textwidth}
  \centering
  \includegraphics[width=.95\textwidth]{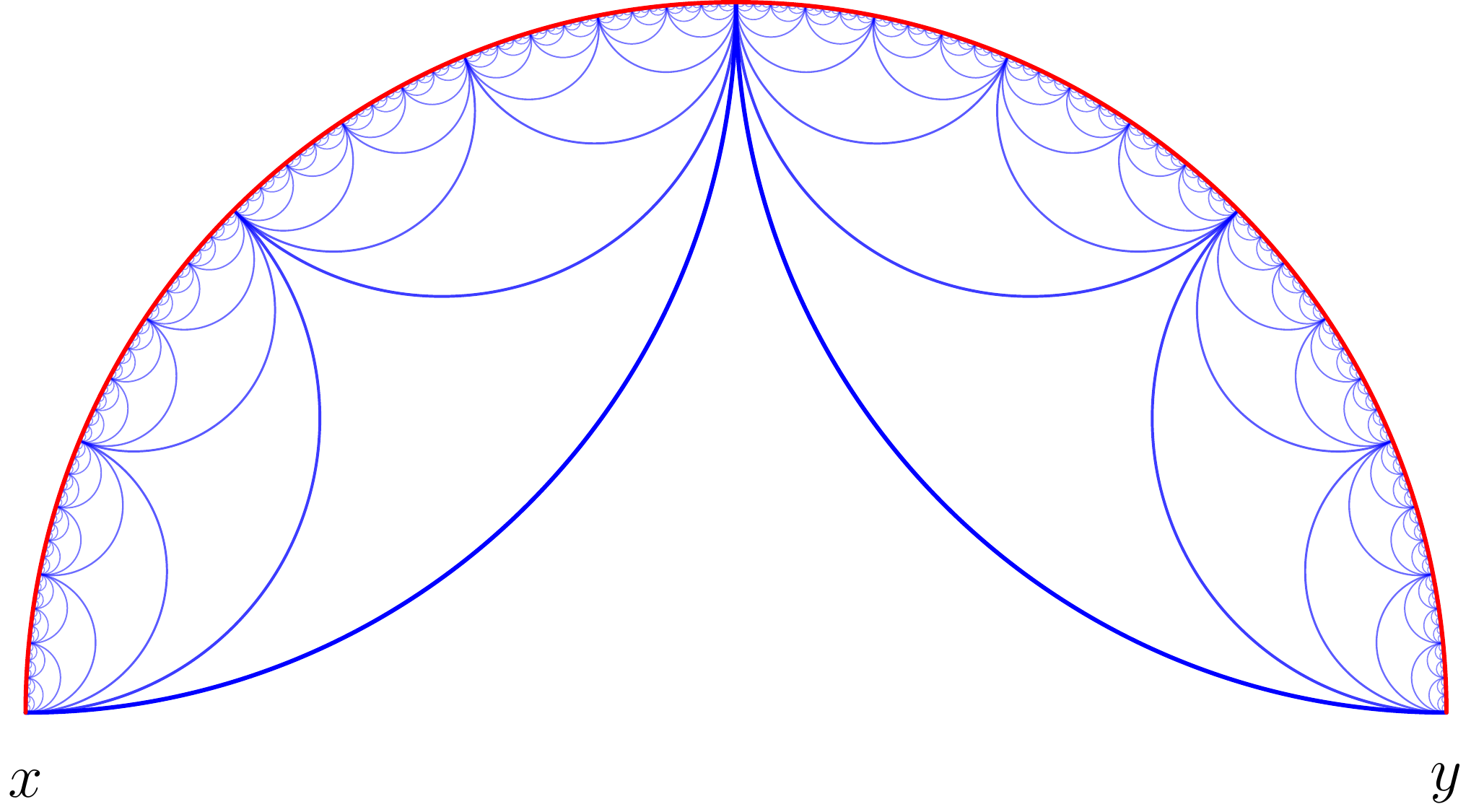}
\end{minipage}
\caption{On the left: The red edges form $M$, and together with the blue edges they form $G^\ast$.
On the right: $\hFG-E(\hFG_0)$ with blue edge set and red vertex set.}
\label{fig:WhirlFarey}
\end{figure}

\begin{lemma}\label{findinghFGinWhirl}
There exists an isomorphism $G^\ast\!/M\cong\hFG-E(\hFG_0)$ that associates $0$ and $1$ with the two vertices of $\hFG_0$.
\end{lemma}

\begin{proof}
Let $x$ and $y$ be the two endvertices of $\hFG_0$.

On the one hand, for every $n\in\N$ the $x$--$y$ Hamilton path of $\hFG_n$ formed by the blue edges of $\hFG_n$ induces a linear ordering on the vertex set of $\hFG_n$ in which $x$ is the least element and $y$ is the greatest, and these orderings are compatible for distinct numbers $n$.

On the other hand, for every integer $n\ge 1$ the vertex set $\{0,1\}\cup M_{\le n}$ of $G^\ast_{\le n}/M$ inherits the linear ordering ${\le_n}$ from $\Q$ in which $0<_n e<_n 1$ for all $e\in M_{\le n}$ and $\{u,v\}<_n \{s,t\}$ if and only if $\min \{u,v\}<_{\Q} \min \{s,t\}$ for all $uv\neq st\in M_{\le n}$, and again these linear orderings are compatible for distinct numbers $n$.

An induction on $n\ge 1$ shows that the unique order-isomorphism  $\varphi_n$ between the linearly ordered finite vertex sets of $G^\ast_{\le n}/M$ and $\hFG_n-E(\hFG_0)$ is a graph-isomorphism such that $\varphi_1\subset\cdots\subset\varphi_n$.
Then the ascending union $\bigcup_{n=1}^\infty\varphi_n$ of these isormorphisms is the desired graph-isomorphism between $G^\ast\!/M$ and $\hFG-E(\hFG_0)$ that associates $0$ and $1$ with the two vertices of $\hFG_0$.
\end{proof}

In order to find a Farey graph minor in $G$, we must find two halved Farey graph minors in $G$.
For this, we consider copies of $G^\ast$ in $G$ that arise by linear transformation.
Every permutation $\pi$ of $\Q$ acts on both the set of graphs $H$ with $V(H)\subset\Q$ and the set of edge sets $D$ with $D\subset [\Q]^2$ by renaming every vertex $v$ to~$\pi(v)$.
Then we write $\pi H$ and $\pi D$ for the resulting graph and edge set.
Now let us consider the two permutations $\pi_1(x):=(1/9)x+3/9$ and $\pi_2(x):=(1/9)x+5/9$.
These send $G^\ast$ to copies $\pi_1 G^\ast$ and $\pi_2 G^\ast$ of $G^\ast$ that are subgraphs of~$G$.
By Lemma~\ref{findinghFGinWhirl} we have $\pi_1 G^\ast\!/\pi_1 M\cong\hFG-E(\hFG_0)$ and $\pi_2 G^\ast\!/\pi_2 M\cong\hFG-E(\hFG_0)$ by isomorphisms that associate the vertices $3/9$, $4/9$ and $5/9$, $6/9$ with the two vertices of $\hFG_0$.
Joining these two halved Farey graph minors appropriately yields the desired Farey graph minor, as shown in Figure~\ref{fig:WhirlFareyMinor}:


\begin{figure}[ht]
    \centering
    \includegraphics[width=.618\textwidth]{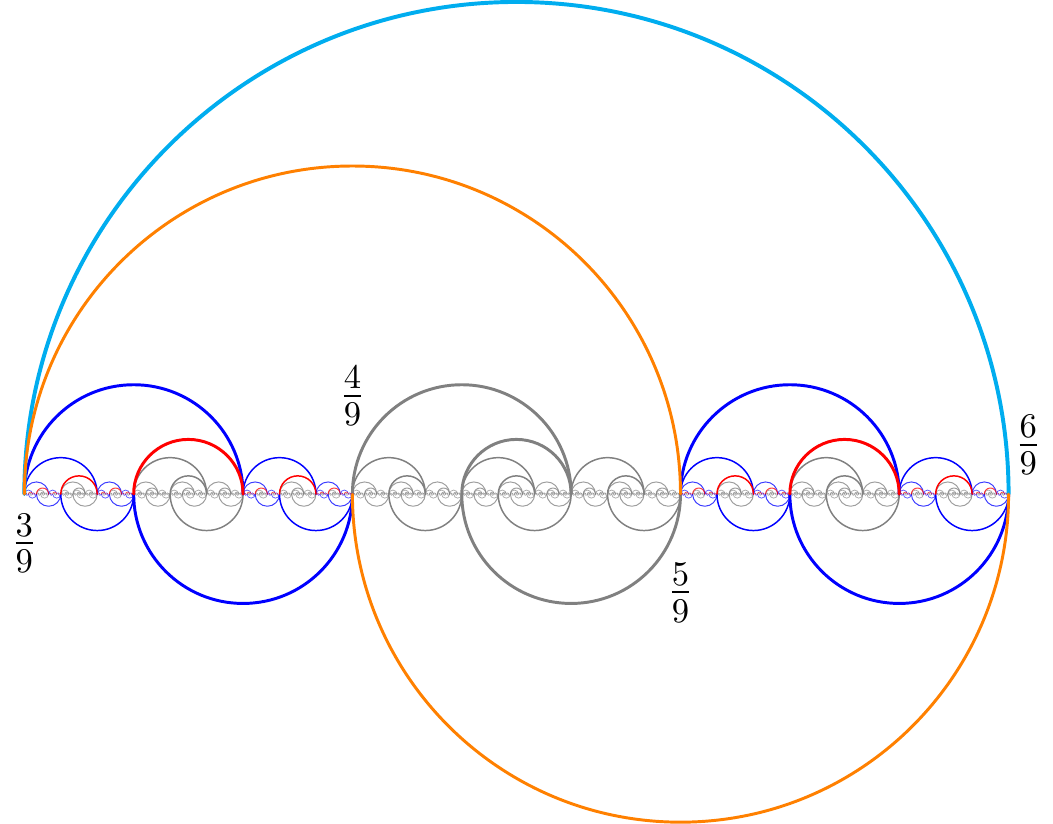}
    \caption{This is $G[\,V\cap[3/9,6/9]\,]$. The two subgraphs $\pi_1 G^\ast$ and $\pi_2 G^\ast$ are drawn using both red and blue, like in Figure~\ref{fig:WhirlFarey}. Theorem~\ref{FareyGraphCantor} states that the Farey graph arises from the subgraph consisting of the coloured edges by contracting red and orange while keeping blue and cyan.}
    \label{fig:WhirlFareyMinor}
\end{figure}

\begin{mainresult}\label{FareyGraphCantor}
The whirl graph contains the Farey graph as a minor with branch sets of size two:
\begin{alignat*}{3}
    \FG\cong \big(&\pi_1 G^\ast &&{}\cup \pi_2 G^\ast &&{}+\big\{\tfrac{3}{9},\tfrac{5}{9}\big\}+\big\{\tfrac{4}{9},\tfrac{6}{9}\big\}+\big\{\tfrac{3}{9},\tfrac{6}{9}\big\}\big)\,\big/\\
    \big(&\pi_1 M &&{}\cup \pi_2 M &&{}+\big\{\tfrac{3}{9},\tfrac{5}{9}\big\}+\big\{\tfrac{4}{9},\tfrac{6}{9}\big\}\big)
\end{alignat*}
where $\pi_1 G^\ast\cup \pi_2 G^\ast+\big\{\tfrac{3}{9},\tfrac{5}{9}\big\}+\big\{\tfrac{4}{9},\tfrac{6}{9}\big\}+\big\{\tfrac{3}{9},\tfrac{6}{9}\big\}\subset G$.
But the whirl graph does not contain the Farey graph as a topological minor.\qed
\end{mainresult}

\end{document}